\newcommand{\drafT}{19 July 2011}
\documentclass[11pt,fleqn,leqno]{article}
\usepackage{amsmath,amssymb,graphicx}
\usepackage{epsfig}
\usepackage{graphics}
\usepackage{enumerate}
\usepackage{a4wide}
\usepackage{url}

\small\normalsize

\newtheorem{theo}{Theorem}[section]
\newtheorem{proposition}[theo]{Proposition}

\newtheorem{conjecture}[theo]{Conjecture}
\newtheorem{corollary}[theo]{Corollary}

\newenvironment{proof}{\begin{trivlist}\item[]\textbf{Proof}\mbox{ \ }}%
       {\qquad\hspace*{\fill}$\Box$\end{trivlist}}
\newenvironment{proofpl}[1]{\begin{trivlist}\item[]\textbf{Proof of
      #1}\mbox{ \ }}%
       {\qquad\hspace*{\fill}$\Box$\end{trivlist}}
\newcommand{\qite}[1]{\noindent\leavevmode\hangindent1.5\parindent%
        \noindent\hbox to1.5\parindent{#1\hss}\ignorespaces}

\newcommand{\dgon}{\mbox{$D$-gon}}
\newcommand{\Eb}{E^B}
\newcommand{\Ep}{E_\phi}
\newcommand{\Eps}[1]{E_{\phi_#1}}
\newcommand{\Eu}{E^U}

\newcommand{\cl}{\mathrm{cl}}
\newcommand{\spn}{\mathrm{cl}}
\newcommand{\spnn}[1]{\mathrm{cl}\bigl(#1\bigr)}
\newcommand{\UC}{\Upsilon_c}
\newcommand{\UF}{\Upsilon\!_f}
\newcommand{\UI}{\Upsilon}

\begin{document}

\title{\textbf{Cyclic Orderings and Cyclic Arboricity of Matroids}}

\author{\quad\\
  Jan van den Heuvel\,$^\ast$ \ and \ St\'ephan
  Thomass\'e\,$^\dagger$\\[3mm]
  $^\ast$\,\emph{Department of Mathematics, London School of
    Economics, London, UK}\\[1mm]
  $^\dagger$\,\emph{Universit\'e Montpellier II - CNRS, LIRMM, 
    Montpellier, France}}

\date{\mbox{}\\[3mm]\drafT}

\maketitle

{\renewcommand{\thefootnote}{\relax}
  \footnotetext{Part of the research for this paper was done during a visit
    of JvdH to LIRMM-Universit\'e Montpellier II\@. JvdH's visit was made
    possible by grants from the Alliance Programme of the British Council
    and from the R\'egion Languedoc-Roussillon.

    \hskip7.5pt email\,: \texttt{jan\,@\,maths.lse.ac.uk},
    \texttt{thomasse\,@\,lirmm.fr}}}

\begin{abstract}
\noindent
We prove a general result concerning cyclic orderings of the elements of a
matroid. \emph{For each matroid~$M$, weight function
  $\omega:E(M)\rightarrow\mathbb{N}$, and positive integer~$D$, the
  following are equivalent.} (1)~\emph{For all $A\subseteq E(M)$, we have
  $\sum_{a\in A}\omega(a)\le D\cdot r(A)$.} (2)~\emph{There is a map~$\phi$
  that assigns to each element~$e$ of~$E(M)$ a set~$\phi(e)$ of~$\omega(e)$
  cyclically consecutive elements in the cycle $(1,2,\ldots,D)$ so that
  each set $\{\,e\mid i\in\phi(e)\,\}$, for $i=1,\ldots,D$, is
  independent.}

As a first corollary we obtain the following. \emph{For each matroid~$M$
  such that~$|E(M)|$ and~$r(M)$ are coprime, the following are equivalent.}
(1)~\emph{For all non-empty $A\subseteq E(M)$, we have
  $|A|/r(A)\le|E(M)|/r(M)$.} (2)~\emph{There is a cyclic permutation
  of~$E(M)$ in which all sets of~$r(M)$ cyclically consecutive elements are
  bases of~$M$.} A second corollary is that the circular arboricity of a
matroid is equal to its fractional arboricity.

These results generalise classical results of Edmonds, Nash-Williams and
Tutte on covering and packing matroids by bases and graphs by spanning trees.

\medskip\noindent
\textbf{Keywords}\,: matroid, base of a matroid, cyclic ordering,
arboricity, circular arboricity
\end{abstract}

\section{Introduction and Results}

\subsection{Cyclic Orderings of Matroids}

We assume the reader is familiar with the basics of matroid theory, as can
be found in, e.g., the book of Oxley~\cite{JO}. All matroids in this paper
are assumed to be finite and without loops. We use~$E$ for the ground set
of a matroid under consideration.

The crucial axiom on the bases of a matroid is the \emph{exchange axiom}.
There are several forms of this (\,see~\cite{JO}\,); the important one for
us is\,: if~$B$ and~$B'$ are bases and $b'\in B'\setminus B$, then there
exists $b\in B\setminus B'$ such that $(B\setminus\{b\})\cup\{b'\}$ is a
base. This axiom implies that given two bases~$B$ and~$B'$, there exists a
sequence of exchanges that transforms~$B$ into~$B'$, i.e., there is a
sequence of bases $B=B_0,\dots,B_i,\dots,B_r=B'$ for which the symmetric
difference of every two consecutive bases has two elements. If~$B$ and~$B'$
are disjoint (\,hence~$r$ is equal to the rank~$r(M)$\,), one way to get
such a sequence is as follows. Let $B'=\{b'_1,b'_2,\dots,b'_r\}$. Take
$b_1\in B$ such that $(B\setminus\{b_1\})\cup\{b'_1\}$ is a base, then take
$b_2\in B$ such that $(B\setminus\{b_1,b_2\})\cup\{b'_1,b'_2\}$ is a base,
etc. We obtain a sequence $(b_1,b_2,\dots,b_r,b'_1,b'_2,\dots,b'_r)$ in
which every set of~$r$ consecutive elements forms a base (\,indeed, these
can be taken as the sets~$B_i$\,).

Studying the structure of symmetric exchanges in matroids, Gabow~\cite{Gab}
asked if it is possible to choose the sequence
$(b_1,\dots,b_r,b'_1,\dots,b'_r)$ so that it provides a \emph{cyclic
  ordering} in which each~$r$ cyclically consecutive elements form a base;
i.e., the sequences $(b'_2,\dots,b'_r,b_1)$, $(b'_3,\dots,b'_r,b_1,b_2)$,
etc., form bases as well. This question was raised again by
Wiedemann~\cite{DW84} and formulated as a conjecture by Cordovil and
Moreira~\cite{CM}.

\begin{conjecture}[\,\cite{CM,Gab,DW84}\,]\label{c:2base2}\mbox{}\\*
  Let $B=\{b_1,\dots,b_r\}$ and $B'=\{b'_1,\dots,b'_r\}$ be two bases of a
  matroid. There is a permutation $(b_{\pi(1)},\ldots,b_{\pi(r)})$ of the
  elements of~$B$ and a permutation $(b'_{\pi'(1)},\ldots,b'_{\pi'(r)})$ of
  the elements of~$B'$ such that the combined sequence
  $(b_{\pi(1)},\ldots,b_{\pi(r)},b'_{\pi'(1)},\ldots,b'_{\pi'(r)})$ is a
  cyclic ordering in which every~$r$ cyclically consecutive elements form a
  base.
\end{conjecture}

\noindent
Conjecture~\ref{c:2base2} has been proved for graphical matroids
\cite{CM,KUM88,DW84}.

A possible easier conjecture is that a suitable cyclic ordering can be
obtained by permuting all elements in the union of the two bases.

\begin{conjecture}\label{c:2base}\mbox{}\\*
  Given two bases $B=\{b_1,\dots,b_r\}$ and $B'=\{b'_1,\dots,b'_r\}$ of a
  matroid, there is a permutation of the sequence
  $(b_1,b_2,\dots,b_r,b'_1,b'_2,\dots,b'_r)$ in which every~$r$ cyclically
  consecutive elements form a base.
\end{conjecture}

\noindent
Since Conjecture~\ref{c:2base2} is known to hold for graphical matroids, so
does this weaker conjecture.

It is obvious that the linear ordering from the introductory paragraphs
exists for any number of bases. Hence a natural generalisation of the
previous conjectures is to start with $k\ge2$ bases and require a suitable
cyclic ordering of the elements of these~$k$ bases combined. No results are
known for $k\ge3$, not even for graphical matroids.

Kajitani \emph{et al.}~\cite{KUM88} formulated the following much more
general conjecture.

\begin{conjecture}[\,\cite{KUM88}\,]\label{c::main1}\mbox{}\\*
  For a loopless matroid~$M$, there is a cyclic ordering of~$E$ such that
  every $r(M)$ cyclically consecutive elements are bases of~$M$ if and only
  if the following condition is satisfied\,:
  \begin{equation}\label{cond1}
    \text{for all non-empty $A\subseteq E$, we have
      $\dfrac{|A|}{r(A)}\:\le\:\dfrac{|E|}{r(M)}\hskip1pt$.}
  \end{equation}
\end{conjecture}

\noindent
Following Catlin \emph{et al.}~\cite{CGHL}, matroids that satisfy
condition~\eqref{cond1} are called \emph{uniformly dense}.

The fact that~\eqref{cond1} is necessary for the required cyclic ordering
of the ground set to exist was already observed by Kajitani \emph{et
  al.}~\cite{KUM88}. They also proved Conjecture~\ref{c::main1} for the
cycle matroids of some special classes of graphs.

A matroid admitting a partition of its ground set into bases is uniformly
dense. Hence Conjecture~\ref{c::main1} implies Conjecture~\ref{c:2base}
(\,and the generalisation of this conjecture with an arbitrary number of
bases\,). Conjecture~\ref{c::main1} also implies that if~$M$ is uniformly
dense and $|E|/r(M)=P/Q$ (\,$P,Q\in\mathbb{N}$\,), then there exist~$P$
bases such that each element of~$E$ appears in exactly~$Q$ of them. This
weaker result was proved by Catlin \emph{et al.}~\cite{CGHL} and Fraisse
and Hell~\cite{FH}.

We prove Conjecture~\ref{c::main1} for a special class of matroids.

\begin{theo}\label{t::main1}\mbox{}\\*
  Let~$M$ be a loopless matroid such that~$|E|$ and~$r(M)$ are coprime
  (\,i.e., $\gcd(|E|,r(M))=1$\,). There is a cyclic ordering of~$E$ such
  that every $r(M)$ cyclically consecutive elements are bases of~$M$ if and
  only if~\eqref{cond1} holds.
\end{theo}

\noindent
The proof of this theorem can be found in Section~\ref{sec-main1}. It
follows from a more technical result we formulate next.

A \emph{weighted matroid $(M,\omega)$} is a matroid~$M$ together with a
weight function $\omega:E\rightarrow\mathbb{Q}^+$. The weight~$\omega(A)$
of a subset $A\subseteq E$ is the sum of the weights of its elements. For a
positive real number~$d$, let~$S_d$ be the circle with circumference~$d$
(\,interpreted as the interval $[0,d]$ with ends identified, or
equivalently as the quotient $\mathbb{R}/d\,\mathbb{R}$\,). When speaking
about a (\,left-closed, right-open\,) \emph{cyclic interval $[x,y)$}
of~$S_d$, we interpret it as the part of the circle that starts at~$x$ and
follows~$S_d$ in the positive direction until reaching~$y$.

Let $(M,\omega)$ be a weighted matroid and~$\phi$ a mapping from~$E$
to~$S_d$. This mapping associates to every $e\in E$ a cyclic interval
$[\,\phi(e),\,\phi(e)+\omega(e)\,)$ of~$S_d$. Conversely, to every
point~$x$ of~$S_d$, we can associate the set
$\Ep(x)=\{\,e\in E\mid x\in[\,\phi(e),\,\phi(e)+\omega(e)\,)\,\}$.

We are interested in mappings $\phi:E\rightarrow S_d$ such that~$\Ep(x)$ is
independent for every point~$x$ in~$S_d$. It is obvious that for large
enough~$d$ such a mapping always exists. Our main result gives the exact
lower bound on~$d$ for which such a mapping is possible.

\begin{theo}\label{t::cyclickhi}\mbox{}\\*
  Let $(M,\omega)$ be a loopless weighted matroid and~$d$ a positive
  rational number. There exists a mapping $\phi:E\rightarrow S_d$ such that
  $\Ep(x)$ is independent for every point~$x$ in~$S_d$ if and only if the
  following condition is satisfied\,:
  \begin{equation}\label{cond2}
    \text{for all non-empty $A\subseteq E$, we have
      $d\ge\dfrac{\omega(A)}{r(A)}\hskip1pt$.}
  \end{equation}
\end{theo}

\noindent
The proof of Theorem~\ref{t::cyclickhi} is given in Section~\ref{sec2}. We
first describe some other corollaries of the theorem in the next subsection.

\subsection{Variants of Arboricity of Matroids}

The \emph{arboricity~$\UI(M)$} of a matroid~$M$ is the minimum number of
bases needed to cover all elements of the matroid. Since every base can
contain at most~$r(A)$ elements for any $A\subseteq E$, the arboricity of a
matroid is at least
$\max\limits_{\varnothing\ne A\subseteq E}\dfrac{|A|}{r(A)}\hskip1pt$. Call
this maximum the \emph{maximal density~$\gamma(M)$}. (\,Notice that~$M$ is
uniformly dense if and only if $\gamma(M)=|E|/r(E)$.\,)

A classical result of Edmonds~\cite{Edm1}, extending the result for graphs
by Nash-Williams~\cite{NW2}, guarantees that this lower bound gives the
right answer.

\begin{theo}[\,Edmonds~\cite{Edm1}\,]\mbox{}\\*
  For a loopless matroid~$M$ we have $\UI(M)=\lceil\gamma(M)\rceil$.
\end{theo}

\noindent
The \emph{fractional arboricity~$\UF(M)$} of a matroid~$M$ is defined as
follows. To each base~$B$ of~$M$ assign a real value $x(B)\ge0$, such that
$\sum_{B\ni e}x(B)\ge1$ for all $e\in E$. Then~$\UF(M)$ is the minimum of
$\sum_{B\in\mathcal{B}(M)}x(B)$ we can obtain under these conditions. Again
we have that~$\UF(M)$ is at least the maximal density~$\gamma(M)$, but as
has been observed by several authors (\,see, e.g., Catlin \emph{et
  al.}~\cite{CGHL} and Scheinerman and Ullman \cite[Section~5.4]{SU}\,), it
follows easily from Edmonds' theorem mentioned above that we have equality.

\begin{proposition}\label{prop-denfr}\mbox{}\\*
  For a loopless matroid~$M$ we have $\UF(M)=\gamma(M)$.
\end{proposition}

\noindent
We now define a third kind of arboricity, the \emph{circular
  arboricity~$\UC(M)$}, introduced by Gon\c{c}alves~\cite{Gon}. As before,
let~$S_d$ be the circle with circumference~$d$. Given a matroid~$M$, we
want to map the elements of~$E$ to~$S_d$ so that for every cyclic unit
interval $[\,x,x+1\,)$, the elements mapped to that cyclic interval form an
independent set. Define~$\UC(M)$ as the infimum over the values of~$d$ for
which such a mapping is possible. Since we assume the matroid to be finite
and loopless, it is easy to see that this infimum is actually attained and
is a rational number.\footnote{ We would have the same definition if cyclic
  intervals were open on both sides or left-open, right-closed. With closed
  cyclic unit intervals we get the same value for the circular arboricity,
  but it would be a real infimum in that instance.}

The definition of the circular arboricity mimics that of the \emph{circular
  chromatic number} of a graph. A \emph{stable set} of a graph is a vertex
set in which no pair is adjacent. The minimum number of stable sets to
cover the vertex set of~$G$ is the \emph{chromatic number~$\chi(G)$}, the
fractional variant is the \emph{fractional chromatic number~$\chi_f(G)$},
and the circular variant (\,the minimum~$d$ such that the vertices of~$G$
can be mapped to~$S_d$ so that the elements mapped to any cyclic unit
interval $[\,x,x+1\,)$ form a stable set\,) is the \emph{circular chromatic
  number~$\chi_c(G)$}. See, e.g., Zhu~\cite{Zhu,Zhu2} for results on this
last parameter (\,including other ways to define it\,).

The following result mimics well-known relations between fractional,
circular and integral chromatic number of graphs. For completeness, we give
its proof in Section~\ref{sec-main2}.

\begin{proposition}\label{prop-crarb}\mbox{}\\*
  For a loopless matroid~$M$ we have $\UF(M)\le\UC(M)\le\UI(M)$ and
  $\UI(M)=\lceil\UC(M)\rceil$.
\end{proposition}

\noindent
It is well-known that the difference between the fractional chromatic
number and the integral chromatic number of a graph~$G$ can be arbitrarily
large (\,see, e.g., Scheinerman and Ullman~\cite[Chapter~3]{SU}\,). Because
$\chi(G)=\lceil\chi_c(G)\rceil$, the same holds for the difference between
the fractional chromatic number and the circular chromatic number.
It is for that reason somewhat surprising that the fractional and circular
arboricity of matroids are always equal. 

\begin{theo}\label{t::main2}\mbox{}\\*
  For a loopless matroid~$M$ we have $\UC(M)=\UF(M)=\gamma(M)$.
\end{theo}

\noindent
This result was conjectured for graphical matroids by Gon\c{c}alves
\cite[Section~3.8]{Gon}. We give the short derivation from
Theorem~\ref{t::cyclickhi} in Section~\ref{sec-main2}.

\section{Proof of Theorem~\ref{t::cyclickhi}}\label{sec2}

We use the notation and conventions from the first section. Since the
weights~$\omega(e)$ and the number~$d$ in the hypothesis of
Theorem~\ref{t::cyclickhi} are assumed to be rational, it is clear that we
can restrict ourselves to mappings from~$E$ to the rational elements
of~$S_d$. This also shows that the theorem is equivalent to
Theorem~\ref{t::cychi-a} below.

A \emph{\dgon}, for a positive integer~$D$, is the sequence
$(1,2,\ldots,D)$ in cyclic order; in other words\,: the integer elements of
the circle~$S_D$. For simplicity, we use $[D]=\{1,\ldots,D\}$ for the
elements of the \dgon, but we must remain aware of the cyclic structure of
the \dgon. In particular, for a mapping $\phi:E\rightarrow[D]$ the cyclic
interval $[\,\phi(e),\,\phi(e)+\omega(e)\,)$ corresponds to the sequence of
integers $\bigl(\phi(e),\phi(e)+1,\ldots,\phi(e)+\omega(e)-1\bigr)$, taken
modulo~$D$.

\begin{theo}\label{t::cychi-a}\mbox{}\\*
  Let $(M,\omega)$ be a loopless weighted matroid with non-negative integer
  weights and~$D$ a positive integer. The following statements are
  equivalent.

  \smallskip
  \qite{\rm(a)}There exists a mapping $\phi:E\rightarrow[D]$ such that for
  every $x\in[D]$, the set
  $\{\,e\in E\mid\linebreak[2]x\in[\,\phi(e),\,\phi(e)+\omega(e)\,)\,\}$ is
  independent.

  \smallskip
  \qite{\rm(b)}For all $A\subseteq E$, we have $\omega(A)\le D\cdot r(A)$.
\end{theo}

\noindent
In the remainder we often use $A-e$ for $A\setminus\{e\}$, and $A\cup e$
for $A\cup\{e\}$.

An essential tool in our proof is the closure operator for matroids. In
particular, we will use repeatedly that if $e\in A\subseteq E$, then
$e\in\spn(A-e)$ if and only if~$e$ is contained in a circuit of~$A$. The
set~$A$ \emph{spans~$E$} if $\cl(A)=E$.

\begin{proofpl}{Theorem~\ref{t::cychi-a}}
  For a mapping $\phi:E\rightarrow[D]$ and $e\in E$, write
  $J_\phi(e)=[\,\phi(e),$ $\phi(e)+\omega(e)\,)$. Recall the notation
  $\Ep(x)=\{\,e\in E\mid x\in J_\phi(e)\,\}$, for any $x\in[D]$.

  Suppose first that a mapping~$\phi$ satisfying~(a) exists. For a set
  $A\subseteq E$, count the pairs $(a,x)$ with $a\in A$ and
  $x\in J_\phi(a)$ in two ways. Since each~$J_\phi(a)$ contains~$\omega(a)$
  elements from~$[D]$, there are~$\omega(A)$ such pairs. On the other hand,
  for each $x\in[D]$ we have that $\Ep(x)$ is independent, hence the number
  of $a\in A$ with $x\in J_\phi(a)$ is at most~$r(A)$. This gives that
  there are at most $D\cdot r(A)$ pairs, proving that~(b) holds.

  So we are left to prove (b)\:$\Rightarrow$\:(a). For this, let~$D$
  satisfy the condition in~(b). We prove~(a) by induction on~$|E|$. (\,It
  is trivially true if $|E|=1$.\,)

  If there is an $e\in E$ such that $\omega(e)=0$, then we can remove~$e$
  from the matroid and are done by induction. So we can assume
  $\omega(e)>0$ for all $e\in E$.

  By~(b), $\omega(e)\le D$ for all $e\in E$. Suppose there is an $e\in E$
  with $\omega(e)=D$. Let~$M'$ be the contraction $M/e$, with rank
  function~$r'$ and ground set $E'=E-e$. Since $\omega(e')>0$ for all
  $e'\in E$, (b) guarantees $r(\{e,e'\})=2$ for all $e'\ne e$. Thus~$M'$ is
  loopless. For all $A'\subseteq E'$, we have
  $\omega(A')=\omega(A'+e)-D\le D\cdot r(A'+e)-D=D\cdot (r'(A')+1)-D=
  D\cdot r'(A')$. So we can apply the induction hypothesis on~$M'$\,: there
  is a mapping $\phi:E'\rightarrow[D]$ such that
  $\{\,e'\in E'\mid x\in[\,\phi(e'),\,\phi(e')+\omega(e')\,)\,\}$ is
  independent for every $x\in[D]$. Extend~$\phi$ to~$M$ by setting
  $\phi(e)=1$ (\,or any other element of~$[D]$\,). It is easy to check
  that~$\phi$ satisfies~(a).

  So from now on we assume $1\le\omega(e)\le D-1$ for all $e\in E$.

  Given two mappings $\phi,\phi'$ of~$E$ to the \dgon, we say
  that~\emph{$\phi$ is better than~$\phi '$ }if for every $x\in[D]$,
  $\spnn{E_{\phi'}(x)}\subseteq\spnn{\Ep(x)}$. We also say
  that~\emph{$\phi$ is strictly better than~$\phi'$} if the inclusion is
  strict for some~$x$; while~$\phi$ is \emph{best possible} if no other
  mapping is strictly better.

  Since there are only finitely many mappings to the \dgon, we can choose a
  best possible mapping~$\phi$. Our goal is to prove that~$\phi$
  satisfies~(a) in the theorem. Assume that this is not the case. So there
  is some $x\in[D]$ for which~$\Ep(x)$ is not independent, i.e.,
  $|\Ep(x)|>r(\Ep(x))$, which also means that~$\Ep(x)$ contains a circuit.
  Since $\omega(E)\le D\cdot r(E)$, at the same time there must be a point
  $x'\in[D]$ for which $r(\Ep(x'))<r(E)$, hence~$\Ep(x')$ does not
  span~$E$.

  For an element $e\in E$, a \emph{push of~$e$} consists of
  replacing~$\phi(e)$ by $\phi(e)+1$ (\,modulo~$D$\,) (\,although
  intuitively it is probably more useful to think of it as replacing the
  cyclic interval $[\,\phi(e),\,\phi(e)+\omega(e)\,)$ by
  $[\,\phi(e)+1,\,\phi(e)+\omega(e)+1\,)$\,). We call~$e$ \emph{pushable}
  if~$e$ belongs to a circuit in $\Ep(\phi(e))$. If~$e$ is pushable, then a
  push of~$e$ always results in a better mapping, since the closure of
  $\Ep(\phi(e))$ does not decrease. Moreover, in that case a push gives a
  strictly better mapping if adding~$e$ to $\Ep(\phi(e)+\omega(e)+1)$ does
  increase the closure of that set, i.e., if~$e$ does not belong to a
  circuit of $\Ep(\phi(e)+\omega(e)+1)\cup e$. As~$\phi$ is assumed to be
  best possible, no sequence of pushes should result in a strictly better
  mapping.

  On the other hand, there always are pushable elements. This follows from
  our earlier observation that for some $x\in[D]$, there is a circuit~$C$
  in~$\Ep(x)$. Going back (\,in negative direction\,) along the \dgon,
  starting from~$x$, let~$y$ be the last point for which
  $C\subseteq\Ep(y)$. (\,Such a point must exists, since $\omega(e)\le D-1$
  for all $e\in E$.\,) By the choice of~$y$, there exists $e\in C$ such
  that $\phi(e)=y$; such an~$e$ is pushable.

  {}From now on we assume that we only push elements that are pushable.
  {}From the arguments in the previous paragraphs, there exists an infinite
  sequence of pushes. We make the sequence of pushes deterministic as
  follows. Start with some initial ordering $e_1,\ldots,e_m$ of the
  elements. Every time we push an element, rearrange the ordering by moving
  the pushed element to the back of the sequence. (\,So elements towards
  the end of the ordering have been pushed ``more recently'' than those
  towards the beginning.\,) Whenever we have a choice between pushable
  elements, we always push the first pushable element according to the
  ordering at that moment.

  Considering the deterministic sequence of pushes thus obtained, we call
  $e\in E$ \emph{bounded} if it is pushed a finite number of times;
  otherwise it is \emph{unbounded}. Starting with~$\phi$ this means that
  after a finite number of pushes we obtain a mapping for which all bounded
  elements have reached their final position on the $D$-gon. Continuing
  with the sequence, the sequence of mappings eventually must become
  periodic, say with period~$T$. Let $\phi_1,\ldots,\phi_T$ be the mappings
  occurring in this periodic sequence. We analyse the properties of this
  sequence in some detail.

  We first remark that by the definition of pushable and the assumption
  that~$\phi$ is best possible, each~$\phi_i$ is also best possible. This
  guarantees the following.
  \begin{trivlist}\item[]\textbf{Claim 1}\quad\emph{For all~$i,j$ and all
      $x\in[D]$, $\spnn{\Eps{i}(x)}=\spnn{\Eps{j}(x)}$.}
  \end{trivlist}

  Let~$\Eu$ be the set of unbounded elements and~$\Eb$ the bounded ones. As
  some~$\Eps{1}(x)$ do not span~$E$, $\Eb$ is non-empty. Also, by our
  supposition that we are dealing with an infinite sequence of pushes,
  $\Eu$ is not empty.

  Let~$e$ be an element in~$\Eb$. Set $x_e=\phi_1(e)$. Since~$e$ has
  reached its final position by the time we consider the mappings
  $\phi_1,\ldots,\phi_T$, we have $\phi_i(e)=\phi_1(e)=x_e$ for all~$i$.

  \begin{trivlist}\item[]\textbf{Claim 2}\quad\emph{For all~$i$, $e$ does
      not belong to a circuit of~$\Eps{i}(x_e)$.}
  \end{trivlist}
  Indeed, suppose this is false for some~$i$. Thus~$e$ is pushable
  in~$\phi_i$. Since this holds in each of the (\,infinitely many\,) later
  appearances of~$\phi_i$, eventually~$e$ becomes the first among the
  pushable elements in~$\phi_i$. So~$e$ will eventually be pushed, a
  contradiction.

  \medskip
  By Claim~2, all the pushes of elements from~$\Eps{i}(x_e)$, for any~$i$,
  involve circuits that do not contain~$e$. Using Claim~1 this gives
  $\spnn{\Eps{i}(x_e)-e}=\spnn{\Eps{j}(x_e)-e}$ for all~$i,j$. Similarly,
  $e\notin\spnn{\Eps{i}(x_e)-e}$ for all~$i$.

  Now, additionally, let~$f$ be an element in~$\Eu$. Since~$f$ cycles
  infinitely around the \dgon, there is a~$j$ such that $f\in\Eps{j}(x_e)$.
  But that means trivially that
  $\spnn{(\Eps{j}(x_e)\cup f)-e}=\spnn{\Eps{j}(x_e)-e}$. Using Claim~1 and
  the relations above, this gives for all~$i,j$\,:
  \[\spnn{(\Eps{i}(x_e)\cup f)-e}\:=\:\spnn{(\Eps{j}(x_e)\cup f)-e}\:=\:
  \spnn{\Eps{j}(x_e)-e}\:=\:\spnn{\Eps{i}(x_e)-e}.\]
  Since this holds for all $f\in\Eu$, we obtain
  $\spnn{(\Eps{i}(x_e)\cup\Eu)-e}=\spnn{\Eps{i}(x_e)-e}$. As
  $e\notin\spnn{\Eps{i}(x_e)-e}$ for all~$i$, this gives the following.
  \begin{trivlist}\item[]\textbf{Claim 3}\quad\emph{For all~$i$ and
      $e\in\Eb$, we have $e\notin\spnn{(\Eps{i}(x_e)\cup\Eu)-e}$.}
  \end{trivlist}

  Since $\Eu\subseteq(\Eps{i}(x_e)\cup\Eu)-e$, this immediately leads to
  $e\notin\spn(\Eu)$. But this holds for all $e\in\Eb$, and so
  $\Eb\cap\spn(\Eu)=\varnothing$. We have proved the following claim\,:
  \begin{trivlist}\item[]\textbf{Claim 4}\quad\emph{$\spn(\Eu)=\Eu$.}
  \end{trivlist}

  Next we prove our final claim.
  \begin{trivlist}\item[]\textbf{Claim 5}\quad\emph{For all~$i$ and
      $x\in[D]$, every circuit~$C$ of $\Eps{i}(x)$ is included in~$\Eu$.}
  \end{trivlist}
  For suppose there is an $x\in[D]$ and a circuit~$C$ in $\Eps{i}(x)$ such
  that $C\cap\Eb\ne\varnothing$. Going back (\,in negative direction\,)
  along the \dgon, starting from~$x$, let~$y$ be the last point for which
  $C\cap\Eb\subseteq\Eps{i}(y)$. (\,Here we use again that
  $\omega(e)\le D-1$ for all $e\in E$.\,). By the choice of~$y$, there
  exists $e\in C\cap\Eb$ such that $\phi_i(e)=y$, i.e., $y=x_e$. Since
  $e\in C\subseteq\Eps{i}(x_e)\cup\Eu$, that would give
  $e\in\spnn{(\Eps{i}(x_e)\cup\Eu)-e}$, contradicting Claim~3.

  By Claim~4, the contraction $M/\Eu$ with ground set~$\Eb$ is loopless. By
  Claim~5, we have that $\phi_1|_{\Eb}:\Eb\rightarrow[D]$ satisfies
  condition~(a) for the weighted matroid $(M/\Eu\!,\omega|_{\Eb})$.

  Now consider $M\backslash\Eb$, the submatroid of~$M$ restricted to~$\Eu$,
  and let~$r^U$ be the rank function of this matroid. Since $r^U(A)=r(A)$
  for all $A\subseteq\Eu$, we have that $\omega(A)\le D\cdot r^U(A)$ for
  all $A\subseteq\Eu$. Hence by the induction hypothesis, there exists a
  mapping $\phi^U:E^U\rightarrow[D]$ that satisfies~(a) for the weighted
  matroid $(M\backslash\Eb\!,\omega|_{\Eu})$.

  Combine the mappings $\phi_1|_{\Eb}$ and $\phi'$ to a mapping~$\psi$
  of~$E$ to the \dgon\,:
  \[\psi(e)\:=\:\biggl\{\begin{array}{ll}
    \phi_1|_{\Eb}(e),&\text{if $e\in\Eb$};\\[2pt]
    \phi^U(e),&\text{if $e\in\Eu$}.\end{array}\]
  For each point $x\in[D]$, the set~$E_\psi(x)$ obtained from this mapping
  has the property that $E_\psi(x)\cap\Eb$ is independent in $M/\Eu$
  and $E_\psi(x)\cap\Eu$ is independent in $M\backslash\Eb$. That means the
  whole set~$E_\psi(x)$ is independent in~$M$, proving that~$\psi$
  satisfies condition~(a) for the weighted matroid $(M,\omega)$. This
  completes the proof of the theorem.
\end{proofpl}

\noindent
The next results are just Theorems~\ref{t::cychi-a} and~\ref{t::cyclickhi}
in terms of dual matroids.

\begin{corollary}\label{c::dual-a}\mbox{}\\*
  Let $(M,\omega)$ be a loopless weighted matroid with non-negative integer
  weights and~$D$ a positive integer. The following statements are
  equivalent.

  \smallskip
  \qite{\rm(a)}There exists a mapping $\phi:E\rightarrow[D]$ such that for
  every $x\in[D]$, the set $\Ep(x)$ spans~$E$.

  \smallskip
  \qite{\rm(b)}For all $A\subseteq E$, we have
  $\omega(A)\ge D\cdot(r(E)-r(E\setminus A))$.
\end{corollary}

\begin{proof}
  This follows easily by applying Theorem~\ref{t::cychi-a} to the dual
  matroid~$M^*$ with weight $\omega^*(e)=D-\omega(e)$, for all $e\in E$.
\end{proof}

\noindent
{}From this we can form the dual version of Theorem~\ref{t::cyclickhi}.

\begin{corollary}\mbox{}\\*
  Let $(M,\omega)$ be a weighted matroid and~$d$ a positive rational
  number. There exists a mapping $\phi:E\rightarrow S_d$ such that $\Ep(x)$
  spans~$E$ for every point~$x$ in~$S_d$ if and only if the following
  condition is satisfied\,: for all $A\subseteq E$ with $r(E\setminus
  A)<r(E)$, we have $d\le\dfrac{\omega(A)}{r(E)-r(E\setminus A)}\hskip1pt$.
\end{corollary}

\section{Cyclic Orderings of Matroids}\label{sec-main1}

The main goal of this section is to prove Theorem~\ref{t::main1}. We also
give one corollary.

The following is an equivalent formulation of the theorem.

\begin{theo}\label{t::main1-a}\mbox{}\\*
  Let~$M$ be a loopless matroid of rank~$r$ and with~$m$ elements such that
  $\gcd(r,m)=1$. The following statements are equivalent.

  \smallskip
  \qite{\rm(a)}There exists a cyclic ordering $(e_1,\dots,e_m)$ of the
  elements of~$M$ such that every cyclic interval $(e_i,\dots,e_{i+r-1})$
  of length~$r$ is a base of~$M$.

  \smallskip
  \qite{\rm(b)}For all $A\subseteq E$, we have $r\cdot|A|\le m\cdot r(A)$.
\end{theo}

\begin{proof}
  Suppose first that a cyclic ordering $(e_1,\dots,e_m)$ satisfying~(a)
  exists. Let $B_1,\ldots,B_m$ be the bases obtained from that ordering,
  hence each element of~$M$ appears in~$r$ (\,cyclically\,) consecutive
  bases of $(B_1,\ldots,B_m)$. So for all $A\subseteq E$, we have
  $r\cdot|A|=|B_1\cap A|+|B_2\cap A|+\cdots+|B_m\cap A|\le m\cdot r(A)$,
  proving that~(b) holds.
  
  Next assume that $r\cdot|A|\le m\cdot r(A)$ for all $A\subseteq E$.
  Setting $\omega(e)=r$ for all $e\in E$, and taking $D=m$, we can apply
  Theorem~\ref{t::cychi-a} to conclude that there exists a mapping
  $\phi:E\rightarrow[m]$ such that for every $x\in[m]$, the set
  $\{\,e\in E\mid x\in(\phi(e),\phi(e)+1,\ldots,\phi(e)+r-1)\,\}$ is
  independent. (\,We use the notation from the paragraph preceding
  Theorem~\ref{t::cychi-a}.\,) We prove that if $\gcd(r,m)=1$, this mapping
  gives a cyclic ordering such that every cyclic interval of length~$r$
  forms a base.

  Since for all $e\in E$, the sequences $(\phi(e),\ldots,\phi(e)+r-1)$ have
  the same length~$r$, it follows immediately that for each cyclic interval
  $I_r(x)=(x,x+1,\ldots,x+r-1)$ of length~$r$ on the $m$-gon, the set of
  elements mapped to~$I_r(x)$ forms an independent set. Notice that since
  $r\cdot|E|=m\cdot r(E)$, the number of elements $e\in E$ that are mapped
  to a cyclic interval~$I_r(x)$ on the $m$-gon is exactly~$r$. Since these
  elements form an independent set, there cannot be more than~$r$. And if
  there would be fewer than~$r$ mapped to~$I_r(x)$, then some other cyclic
  interval would have more than~$r$ elements, which is also impossible.

  We have that~$\phi$ is a mapping of the $m$-element set~$E$ to the
  $m$-gon such that every set of~$r$ consecutive points from the $m$-gon
  intersects $\phi(E)$ in~$r$ points. Suppose that $x\notin\phi(E)$ for
  some $x\in[m]$. Then the $r-1$ points on the $m$-gon consecutive to~$x$
  contain~$r$ elements of $\phi(E)$, hence $x+r\notin\phi(E)$. Repeating
  this argument, we also have $x+2\,r,x+3\,r,\ldots\notin\phi(E)$. Since
  $\gcd(m,r)=1$, this would mean that $\phi(E)$ is empty, a contradiction.
  Thus $x\in\phi(E)$ for all $x\in[m]$. It follows that~$\phi$ is a
  bijection, and hence~$\phi$ corresponds to an ordering of~$E$ along the
  $m$-gon, This is exactly the cyclic ordering we were looking for.
\end{proof}

\noindent
An immediate corollary is the following, also conjectured, without the
condition $\gcd(w,m)=1$, by Kajitani \emph{et al.}~\cite{KUM88}.

\begin{corollary}\label{cor3}\mbox{}\\*
  Let~$M$ be a loopless matroid with~$m$ elements. Suppose~$w$ is a
  positive integer such that $\gcd(w,m)=1$ and such that for all
  $A\subseteq E$, we have $w\cdot|A|\le m\cdot r(A)$. There exists a cyclic
  ordering $(e_1,\dots,e_m)$ of~$E$ such that every cyclic interval
  $(e_i,\dots,e_{i+w-1})$ of~$w$ elements is independent.
\end{corollary}

\begin{proof}
  Let~$M_w$ be the matroid whose independent sets are the independent sets
  of~$M$ with at most~$w$ elements. Then~$M_w$ is a matroid of rank~$w$,
  and the corollary follows from Theorem~\ref{t::main1-a}.
\end{proof}

\section{Circular Arboricity of Matroids and Related
  Results}\label{sec-main2}

In this final section we prove Theorem~\ref{t::main2}, settling a
conjecture made for graphs by Gon\c{c}alves \cite[page~140]{Gon}. We first
give the short proof of Proposition~\ref{prop-crarb}. As mentioned earlier,
this proof mimics the similar relations for the different types or
chromatic number of graphs.

\begin{proofpl}{Proposition~\ref{prop-crarb}}
  Set $c=\UI(M)$ and take bases $B_1,\ldots,B_c$ covering the elements
  of~$M$. By removing multiple occurrences of an element, we find~$c$
  disjoint independent sets $I_1,\ldots,I_c$. Now for each $e\in E$, if
  $e\in I_i$, then map~$e$ to the point~$i$ on the circle~$S_c$. Thus every
  cyclic unit interval contains exactly one of the independent sets~$I_i$.
  This proves $\UC(M)\le c$.

  Next set $d=\UC(M)$. Suppose $\phi:E\longrightarrow S_d$ satisfies the
  requirements for the circular arboricity. By going round the circle, this
  mapping gives an ordering $(e_1,\ldots,e_m)$ of~$E$ (\,by taking
  $(e_1,\ldots,e_m)$ such that
  $0\le\phi(e_1)\le\phi(e_2)\le\cdots\le\phi(e_m)<d$, braking ties
  arbitrarily\,). With each element~$e_i$ we associate an independent set
  $I_i=\{\,e\in E\mid\phi(e)\in[\,\phi(e_i),\,\phi(e_i)+1\,)\,\}$. If
  necessary, add arbitrarily chosen extra elements to extend~$I_i$ to a
  base~$B_i$. For $i=1,\ldots,m$, let~$x(B_i)$ be equal to the difference
  between $\phi(e_{i+1})$ and $\phi(e_i)$ (\,measuring by going from
  $\phi(e_i)$ to $\phi(e_{i+1})$ in positive direction\,), where we take
  $e_{m+1}=e_1$. Note that~$x(B_i)$ is zero if $\phi(e_i)=\phi(e_{i+1})$.
  For all other bases~$B$ of~$M$, set $x(B)=0$. It is easy to check that
  for all $e\in E$, we have $\sum\limits_{B\ni e}x(B)\ge1$, as well as
  $\sum\limits_{B\in\mathcal{B}(M)}x(B)=d$. This proves $\UF(M)\le d$.

  Continuing from the previous paragraph, take $d_0=\lceil d\rceil$. For
  $i=1,\ldots,d_0-1$ set $I'_i=\{\,e\in E\mid\phi(e)\in[\,i-1,\,i\,)\,\}$,
  while $I'_{d_0}=\{\,e\in E\mid\phi(e)\in[\,d_0-1,\,d\,)\,\}$. Then
  $I'_1,\ldots,I'_{d_0}$ is a collection of independent sets covering~$E$.
  We can extend these sets to bases covering~$E$, showing that
  $\lceil\UC(M)\rceil\ge\UI(M)$. Since we saw already $\UC(M)\le\UI(M)$,
  and as $\UI(M)$ is an integer, we must have $\UI(M)=\lceil\UC(M)\rceil$.
\end{proofpl}

\begin{theo}\label{t::main2-a}\mbox{}\\*
  For a loopless matroid~$M$ we have $\gamma(M)=\UF(M)=\UC(M)$.
\end{theo}

\begin{proof}
  By Propositions~\ref{prop-denfr} and~\ref{prop-crarb} it is enough to
  prove that $\UC(M)\le\gamma(M)$. Take positive integers~$P,Q$ such that
  $\gamma(M)=\dfrac{P}{Q}$. Give weight $\omega(e)=Q$ to all $e\in E$. For
  all $A\subseteq E$ we have $|A|\le\gamma(M)\cdot r(A)$, hence
  $\omega(A)\le P\cdot r(A)$. By Theorem~\ref{t::cychi-a} there is a
  mapping~$\phi$ of~$E$ to the $P$-gon such that for every $x\in[P]$, the
  set $\{\,e\in E\mid x\in[\,\phi(e),\,\phi(e)+Q-1\,)\,\}$ is independent.
  That is equivalent to saying that for every point~$x$ of the $P$-gon,
  $\{\,e\in E\mid\phi(e)\in[\,x,\,x+Q-1\,)\,\}$ is independent. Now we
  define the mapping $\psi:E\rightarrow S_{P/Q}$ by setting
  $\psi(e)=\phi(e)/Q$, for all $e\in E$. Then~$\psi$ has the property that
  for every point~$y$ of~$S_{P/Q}$, the elements of~$E$ mapped to
  $[\,y,y+1\,)$ form an independent set. This shows that $\UC(M)\le
  P/Q=\gamma(M)$ and completes the proof.
\end{proof}

\noindent
We can use exactly the same idea as in the proof above, but using
Corollary~\ref{c::dual-a} instead of Theorem~\ref{t::cychi-a}, to obtain
the following result.

\begin{theo}\label{t::main2-a-dual}\mbox{}\\*
  Let~$M$ be a matroid and~$d$ a real number such that for all
  $A\subseteq E$ with $r(E\setminus A)<r(e)$, we have
  $d\le\dfrac{|A|}{r(E)-r(E\setminus A)}\hskip1pt$. There exists a
  mapping~$\phi$ of~$E$ to the circle~$S_d$ such that for every point~$x$
  of the circle, the set $\{\,e\in E\mid\phi(e)\in[\,x,x+1\,)\,\}$
  spans~$E$.
\end{theo}

\noindent
Because Gon\c{c}alves formulated his original conjecture in terms of
graphs, we give the corollaries of the last two theorems for the case of
graphical matroids. For a graph~$G$, let~$V(G)$ denote the set of vertices,
$E(G)$ the set of edges, and~$c(G)$ the number of components of~$G$.

\clearpage
\begin{corollary}\label{cor4}\mbox{}\\*
  Let~$G$ be a graph and~$d$ a real number such that for every subgraph~$H$
  of~$G$ with at least two vertices, we have
  $d\ge\dfrac{|E(H)|}{|V(H)|-1}\hskip1pt$. There exists a mapping~$\phi$ of
  the edge set~$E(G)$ to the circle~$S_d$ such that for every point~$x$ of
  the circle, $\{\,e\in E(G)\mid\phi(e)\in[\,x,x+1\,)\,\}$ forms an acyclic
  subgraph of~$G$.
\end{corollary}

\begin{corollary}\label{cor5}\mbox{}\\*
  Let~$G$ be a connected graph and~$d$ a positive real number such that for
  every set of edges~$A$ that is a cut in~$G$, we have
  $d\le\dfrac{|A|}{c(G-A)-1}\hskip1pt$. There exists a mapping~$\phi$ of
  the edge set~$E(G)$ to the circle~$S_d$ such that for every point~$x$,
  the set $\{\,e\in E(G)\mid\phi(e)\in[\,x,x+1\,)\,\}$ forms a connected
  spanning subgraph of~$G$.
\end{corollary}

\noindent
Theorem~\ref{t::main2-a} generalises a result of Edmonds~\cite{Edm1}, while
Theorem~\ref{t::main2-a-dual} generalises another result of
Edmonds~\cite{Edm2}. Their graphical versions, Corollaries~\ref{cor4}
and~\ref{cor5}, generalise classical results of Nash-Williams~\cite{NW2},
and of Nash-Williams~\cite{NW1} and Tutte~\cite{Tut}, respectively.

\subsection*{Acknowledgement}

The authors thank the anonymous referees for comments and suggestions that
greatly improved the structure and clarity of the paper. We also like to
thank Michel Goemans for pointing out a serious error in an earlier version
of this paper (\,in which we claimed optimistically to have proved
Conjecture~\ref{c::main1}\,). Similar thanks to L\'aszl\'o V\'egh for
remarks that improved the presentation of the proof of
Theorem~\ref{t::cyclickhi}.

\end{document}